\documentclass[12pt,a4paper]{amsart}
\usepackage{amsmath,amsthm,amssymb,amsfonts,mathabx}
\usepackage{abraces}
\usepackage{tikz-cd}
\usepackage{thmtools,xcolor}
\usepackage[bookmarks=true,hyperindex,pdftex,colorlinks,citecolor=blue,
linkcolor=blue,urlcolor=blue]{hyperref}

\usepackage[english]{babel}
\usepackage{enumitem}

\usepackage{graphicx}

\declaretheoremstyle[
headfont=\color{blue}\normalfont\bfseries,
bodyfont=\color{blue}\normalfont\itshape,
]{colored}

\theoremstyle{plain}
\newtheorem{theorem}{Theorem}[section]
\newtheorem{cor}[theorem]{Corollary}
\newtheorem{corollary}[theorem]{Corollary}
\newtheorem{prop}[theorem]{Proposition}
\newtheorem{lemma}[theorem]{Lemma}

\theoremstyle{definition}

\newtheorem{example}[theorem]{Example}

\newtheorem{remark}[theorem]{Remark}
\newtheorem{definition}[theorem]{Definition}

\newcommand{\R}{\mathbb{R}}
\newcommand{\N}{\mathbb{N}}
\newcommand{\C}{\mathbb{C}}

\newcommand{\W}{\mathcal{W}}
\newcommand{\Lin}{\mathcal{L}}
\newcommand{\K}{\mathcal{K}}

\newcommand{\eps}{\varepsilon}

\newcommand{\vertiii}[1]{{\left\vert\kern-0.25ex\left\vert\kern-0.25ex\left\vert #1
 \right\vert\kern-0.25ex\right\vert\kern-0.25ex\right\vert}}

\DeclareMathOperator{\re}{Re}

\DeclareMathOperator{\NA}{NA}

\DeclareMathOperator{\QNA}{QNA}

\newcommand{\wQNA}{w^*\!\operatorname{QNA}}
\renewcommand{\leq}{\leqslant}
\renewcommand{\geq}{\geqslant}

\title{Weak-star quasi norm attaining operators}

\author[G.~Choi]{Geunsu Choi}
\address[Choi]{Department of Mathematics Education, Dongguk University, Seoul 04620, Republic of Korea \newline
\href{http://orcid.org/0000-0002-4321-1524}{ORCID: \texttt{0000-0002-4321-1524} }}
\email{\texttt{chlrmstn90@gmail.com}}

\author[M.~Jung]{Mingu Jung}
\address[Jung]{School of Mathematics, Korea Institute for Advanced Study, 02455 Seoul, Republic of Korea \newline
\href{http://orcid.org/0000-0003-2240-2855}{ORCID: \texttt{0000-0003-2240-2855} }}
\email{\texttt{jmingoo@kias.re.kr}}
\urladdr{\url{https://clemg.blog/}}

\author[S.~K.~Kim]{Sun Kwang Kim}
\address[Kim]{Department of Mathematics, Chungbuk National University, Cheongju, Chungbuk 28644, Republic of Korea\newline
	\href{http://orcid.org/0000-0002-9402-2002}{ORCID: \texttt{0000-0002-9402-2002}  }}
\email{\texttt{skk@chungbuk.ac.kr}}

\author[M.~Mart\'in]{Miguel Mart\'in}
\address[Mart\'{\i}n]{Universidad de Granada, Facultad de Ciencias.
Departamento de An\'{a}lisis Matem\'{a}tico, 18071-Granada
(Spain)\newline
	\href{http://orcid.org/0000-0003-4502-798X}{ORCID: \texttt{0000-0003-4502-798X} }}
\email{\texttt{mmartins@ugr.es}}
\urladdr{\url{https://www.ugr.es/~mmartins/}}

\date{September 15th, 2022}


\begin{document}
		
\begin{abstract}
For Banach spaces $X$ and $Y$, a bounded linear operator $T\colon X \longrightarrow Y^*$ is said to weak-star quasi attain its norm if the $\sigma(Y^*,Y)$-closure of the image by $T$ of the unit ball of $X$ intersects the sphere of radius $\|T\|$ centred at the origin in $Y^*$. This notion is inspired by the quasi-norm attainment of operators introduced and studied in \cite{CCJM}. As a main result, we prove that the set of weak-star quasi norm attaining operators is dense in the space of bounded linear operators regardless of the choice of the Banach spaces, furthermore, that the approximating operator can be chosen with additional properties. This allows us to distinguish the properties of weak-star quasi norm attaining operators from those of quasi norm attaining operators. It is also shown that, under certain conditions, weak-star quasi norm attaining operators share numbers of equivalent properties with other types of norm attaining operators, but that there are also a number of situations in which they behave differently from the others.
\end{abstract}

\maketitle

\section{Introduction}
In this paper we deal with Banach spaces defined over the scalar field $\mathbb{K} = \R$ or $\C$. Given a Banach space $X$, we write $B_X$ and $S_X$ to denote the closed unit ball and the unit sphere of $X$, respectively. By $X^*$ we denote the topological dual space of $X$. If $Y$ is a Banach space, $\Lin(X,Y)$ denotes the space of all bounded linear operators from $X$ to $Y$ endowed with the operator norm, $\W(X,Y)$ is the closed subspace of weakly compact operators and $\K(X,Y)$ is the closed subspace of compact operators.

Recall that $T\in \Lin(X,Y)$ \emph{attains its norm} (or it is \emph{norm attaining}, writing $T\in \NA(X,Y)$) if there exists $x\in B_X$ such that $\|T\|=\|Tx\|$ (equivalently, if $T(B_X)\cap \|T\|S_Y\neq \emptyset$). The study of norm attaining operators dates back to the pioneer work of J.~Lindenstrauss in 1963 \cite{Lindenstrauss}, who provided the first negative and positive results concerning the denseness of norm attaining operators. In 1977, J.~Bourgain \cite{B2} showed the relation between the denseness of norm attaining operators and the Radon-Nikod\'{y}m property (RNP in short): a Banach space $X$ has the RNP if and only if $\NA(X',Y)$ is dense in $\Lin(X',Y)$ for every Banach space $Y$ and every space $X'$ isomorphic to $X$ (this formulation actually needs a refinement of Bourgain's result due to R.~Huff \cite{Huff}). A counterpart for range spaces does not hold, as one can find a Banach space $X$ such that $\NA(X,\ell_2)$ is not dense in $\Lin(X,\ell_2)$ \cite{Gowers}. There are even compact operators which cannot be approximated by norm attaining ones \cite{martinjfa} and it is unknown whether finite-rank operators always belong to the closure of the set of norm attaining operators. We refer the reader to \cite{Aco-survey, M-RACSAM} for a detailed account on the study of the denseness of norm attaining operators. 

Very recently, a weaker notion of norm attainment has been introduced in  \cite{CCJM} to obtain a characterization of the RNP for range spaces. A bounded linear operator $T\in \Lin(X,Y)$ \emph{quasi attains its norm} (writing $T\in \QNA(X,Y)$) if $\overline{T(B_X)}\cap \|T\|S_Y\neq \emptyset$, equivalently, if there is a convergent sequence $(y_n)\subset T(B_X)$ such that $\lim_n y_n\in \|T\|S_Y$. Of course, we have both $\NA(X,Y)\subseteq \QNA(X,Y)$ and $\mathcal{K}(X,Y)\subseteq \QNA(X,Y)$. It is known that monomorphisms (i.e., bounded linear operators which are bounded below) in $\QNA(X,Y)$ actually belong to $\NA(X,Y)$ \cite[Lemma~2.1]{CCJM}. This fact allows to show that there are Banach spaces $X$ and $Y$ for which $\QNA(X,Y)$ is not dense in $\Lin(X,Y)$. For instance, if $Y$ fails the RNP, there are $X_1$ and $X_2$ isomorphic to $Y$ such that $\QNA(X_1,X_2)$ is not dense in $\Lin(X_1,X_2)$ \cite[Proposition~2.5]{CCJM}. On the other hand, if $Y$ has the RNP, then $\overline{\QNA(X,Y)}=\Lin(X,Y)$ for every Banach space $X$ \cite[Theorem~3.1]{CCJM}.
This provides the characterization of the RNP for range spaces in terms of the quasi norm attainment which is not valid for the usual norm attainment.
Other interesting results on quasi norm attaining operators including characterizations of the reflexivity and the finite-dimensionality can be found in \cite{CCJM}, where the reader may also find some open questions.

Our goal in this paper is to introduce and study a property similar to the quasi norm attainment, but using a coarser topology in the range space, namely the weak-star topology in the case when the range space is a dual space. Here is the notion.

\begin{definition} Let $X$ and $Y$ be Banach spaces.
	We say that an operator $T \in \Lin (X,Y^*)$ \emph{weak-star quasi attains its norm \textup{(}in short, $T \in \wQNA (X,Y^*)$\textup{)}} if there exist a net $(x_\alpha) \subset B_X$ and a vector $u^* \in \|T\| S_{Y^*}$ such that $Tx_\alpha \xrightarrow{w^*} u^*$. In other words, $T \in \wQNA(X, Y^*)$ if and only if $\overline{T (B_X)}^{w^*} \cap \|T \| S_{Y^*} \neq \emptyset$.
\end{definition}

As the weak-star topology depends on the different preduals, the notion above may also be dependent on the range's predual (if there is more than one). A concrete example of this phenomenon can be found in Remark~\ref{remark:depends-on-the-predual}.

In general, we have
\begin{equation}\label{eq:allcontents}
\NA (X, Y^*) \subseteq \QNA (X, Y^*) \subseteq \wQNA(X, Y^*)\subseteq \Lin(X,Y^*).
\end{equation}
We will discuss later on when the above inclusions are actually equalities.

Our first main result here is that, contrary to the case of the norm attainment and the quasi norm attainment, the set $\wQNA(X,Y^*)$ is norm dense in $\Lin(X,Y^*)$ for all Banach spaces $X$ and $Y$ (Theorem~\ref{wQNA_dense}). Actually, a stronger version of this result holds. Namely, the set of weak-star quasi norm attaining operators from $X$ to $Y^*$ whose `transpositions' (the definition of which is given before Proposition \ref{prop:NA-wQNA}) also weak-star quasi attain their norms is always dense in $\mathcal{L} (X, Y^*)$ (Theorem~\ref{theorem:wQNA-dense-2}). On the contrary, there is a pair $(X,Y^*)$ of Banach spaces for which $\QNA(X,Y^*)$ is not dense in $\Lin(X,Y^*)$ (Example~\ref{prop_c_0_Read}). These results are contained in Section~\ref{section:denseness}.

We devote Section~\ref{section:some-remarks} to provide some results on weak-star quasi norm attaining operators. We characterize weak-star quasi norm attaining operators in terms of the norm attainment of the biadjoint operator (Proposition~\ref{prop:basics}) and deduce from this result that the quasi norm attainment and the weak-star quasi norm attainment are equivalent for weakly compact operators (Corollary~\ref{weakly_compact_QNA_wQNA}). More results related to the weak-star quasi norm attainment and its properties with respect to taking adjoint and transposition of operators are also given (Propositions \ref{prop:triple} and \ref{prop:NA_or_QNA_to_wQNA}). Some interesting examples and remarks are also provided, including an explicit example showing that a weak-star quasi norm attaining monomorphism is not necessarily norm attaining (Remark~\ref{remark-monomorphism}), contrary to what happens for the quasi norm attainment.

Section~\ref{section:relations} includes results related to the possible equalities in the chain of inclusions given in \eqref{eq:allcontents}. We first relate the validity of some equalities with the reflexivity of the domain or range space (Remark~\ref{Remark_reflexivity}). We next show that if $Y$ is separable and non-reflexive, the equality $\QNA(\ell_1,Y^{**})=\wQNA(\ell_1,Y^{**})$ implies that $Y$ contains an isomorphic copy of $\ell_1$ (Theorem~\ref{thm:wQNA:range:ell_1}), while the converse does not hold (Example~\ref{counterexample_sep_Y}). There are more related results on the equality between $\wQNA$ and $\QNA$ (Propositions \ref{Grothendieck_WCG} and \ref{wKK_QNA_wQNA}, and Corollary \ref{corollary:LURorc0}). Next, we characterize the finite-dimensionality of a separable Banach space $Y$ by the fact that every operator from $\ell_1$ to $Y^*$ is weak-star quasi norm attaining (Theorem~\ref{Y_sep_wQNA_ell_1_Y*}). As a consequence, we may obtain a result on remotality of weak-star compact sets (Example~\ref{example:remotal}). Finally, we finish the paper by providing an example which shows that the concept of the weak-star quasi norm attainment depends on the choice of predual (Remark~\ref{remark:depends-on-the-predual}).

\section{Denseness results}\label{section:denseness}

The main result of the paper is the following one.

\begin{theorem}\label{wQNA_dense}For arbitary Banach spaces $X$ and $Y$, the set $\wQNA (X, Y^*)$ is dense in $\Lin(X, Y^*)$. Indeed, given $T\in \Lin (X, Y^*)$ and $\eps>0$, there is a rank-one perturbation $S\in \wQNA (X, Y^*)$ of $T$ such that $\|T-S\|<\eps$.
\end{theorem}

To give a proof, we will use the following result from Poliquin and Zizler.

\begin{lemma}[\mbox{\cite[Theorem 1]{PZ}}]\label{PZ_theorem}
Let $X$ be a real Banach space and let $\varphi$ be a $w^*$-lower semicontinuous convex Lipschitz function defined on a $w^*$-compact convex subset $C$ of $X^*$. Given $\eps > 0$, there is a vector $x \in \eps B_X$ such that $\varphi + x$ attains its supremum on $C$ at an extreme point of $C$.
\end{lemma}

\begin{proof}[Proof of Theorem \ref{wQNA_dense}] Let $\eps>0$ be given. Given a nonzero $T \in \Lin (X, Y^*)$, let $C$ be the $w^*$-compact convex set in $Y^*$ given by $C = \overline{T (B_X)}^{w^*}$. By Lemma~\ref{PZ_theorem} and the well known relation between the real and the complex normed spaces (see,  for instance, \cite[p. 55]{FHHMZ}), there exists $y_0 \in Y$ such that $\|y_0\| < \eps/\|T\|$ and that the function $\| \cdot \| + \re y_0$ attains its supremum on $C$ at some extreme point $y_0^* \in C$. Thus,
\begin{equation*}
\|y^*\| + \re y^* (y_0) \leq \|y_0^*\| + \re y_0^* (y_0)
\end{equation*}
for every $y ^* \in C$. It follows that $\re y_0^* (y_0) = |y_0^* (y_0)| \,(=y_0^* (y_0) )$ and that
\begin{equation}\label{theorem:supremum}
\|y^*\| + | y^* (y_0)| \leq \|y_0^*\| + | y_0^* (y_0) |
\end{equation}
for every $y^* \in C$. Define $S \in \Lin (X, Y^*)$ by
\[
S(x) = Tx + [Tx](y_0) \frac{y_0^*}{\|y_0^*\|} \,\, \text{ for each } x \in X.
\]
Then $\|T - S \| \leq \|T \| \|y_0\| < \eps$ and \eqref{theorem:supremum} implies that $\|S \| \leq \|y_0^*\| + |y_0^* (y_0)|$.
Write $z_0^* = \left( 1 + \frac{y_0^* (y_0)}{\|y_0^*\|}\right)y_0^* \in Y^*$, then $\|z_0^*\| = \|y_0^*\| + |y_0^* (y_0)|$. If we take a net $(x_\alpha) \subset B_X$ such that $Tx_\alpha \xrightarrow{w^*} y_0^*$, then
\[
Sx_\alpha = Tx_\alpha + [Tx_\alpha](y_0) \frac{y_0^*}{\|y_0^*\|} \xrightarrow{w^*} y_0^* + \frac{y_0^* (y_0)}{\|y_0^*\|} y_0^* = z_0^*.
\]
This implies that $\|S \| = \|y_0^*\| + |y_0^* (y_0)| = \|z_0^*\|$ and thus $S \in \wQNA (X, Y^*)$.
\end{proof}

If one does not care of the fact that the new operator is a rank-one perturbation of the original one, it is also possible to prove Theorem \ref{wQNA_dense} without using the result of Poliquin and Zizler. More explicitly, the denseness of the set of weak-star quasi norm attaining operators can be obtained by combining the following Proposition \ref{prop:NA-wQNA} with a result of Zizler \cite[Proposition 4]{Z}.

Before we state and prove the result, we require the following notion: given Banach spaces $X$ and $Y$, consider the isometric isomorphism between $\Lin (X, Y^*)$ and $\Lin (Y, X^*)$ given by
\[
T \in \Lin (X, Y^*) \longmapsto T^\circ \in \Lin (Y, X^*),
\]
where
$[T^\circ y ]( x ) =  [T x] (y)$ for every $y \in Y$ and $x \in X$. Let us call $T^\circ$ the \emph{transposition} of the operator $T$.

\begin{prop}\label{prop:NA-wQNA}
Let $X$ and $Y$ be Banach spaces, and let $T \in \Lin(X,Y^*)$. If $T^* \in \NA(Y^{**},X^*)$, then $T^\circ \in \wQNA(Y,X^*)$.
\end{prop}

\begin{proof}
Suppose that $T^*$ attains its norm at $y_0^{**} \in S_{Y^{**}}$. Choose a net $(y_\alpha) \subseteq S_Y$ which converges weak-star to $y_0^{**}$. Then, we have
$$
[T^\circ y_\alpha](x) = [Tx](y_\alpha) \longrightarrow y_0^{**}(Tx) = [T^*y_0^{**}](x)
$$
for every $x \in X$. That is, $T^\circ y_\alpha$ converges weak-star to $T^*y_0^{**}$ with $\|T^*y_0^{**}\| = \|T^\circ\|$, and this concludes the proof.
\end{proof}

\begin{proof}[Alternative proof of the denseness of $\wQNA (X,Y^*)$]
Let $T \in \Lin (X, Y^*)$ and $\eps >0$ be given. By \cite[Proposition 4]{Z}, we can find $R \in \Lin (Y, X^*)$ in such a way that $\|R-T^\circ\| < \eps$ and $R^* \in \NA(X^{**}, Y^*)$. Using Proposition \ref{prop:NA-wQNA}, we have that $R^\circ \in \wQNA (X, Y^*)$. This completes the proof since $\|T - R^\circ \| = \|T^{\circ} - R\| < \eps$.
\end{proof}

As a matter of fact, we can actually prove the following stronger result.

\begin{theorem}\label{theorem:wQNA-dense-2}
Let $X$ and $Y$ be Banach spaces. Then, the set
\[
\bigl\{ T \in \Lin (X,Y^*)\colon T \in \wQNA (X, Y^*) \ \text{ and } \ T^\circ \in \wQNA (Y, X^*)\bigr\}
\]
is dense in $\Lin (X, Y^*)$.
\end{theorem}

We need some notation. Denote by $\Lin(^2X \times Y)$ the Banach space of all continuous bilinear forms $A\colon X \times Y \longrightarrow \mathbb{K}$ equipped with the norm
$$
\|A\| = \sup \{|A(x,y)|\colon (x,y) \in B_X \times B_Y\}.
$$
Note that $\Lin (^2 X \times Y)$ can be identified isometrically with $\Lin (X, Y^*)$ in the canonical way. That is, for $A \in \Lin (^2 X \times Y)$, the associated operator $L_A \in \Lin (X,Y^*)$ is given by $[L_A (x)] (y) = A(x,y)$ for $x \in X$ and $y \in Y$.

\begin{proof}[Proof of Theorem~\ref{theorem:wQNA-dense-2}]
Let $T \in \Lin (X, Y^*)$ and $\eps >0$ be given. Consider the bilinear form $\overline{T} \in \mathcal{L}(^2 X\times Y)$ which corresponds to $T$, that is, $\overline{T} (x, y) = [Tx](y)$ for every $x \in X$ and $y \in Y$. Arguing as in the proof of \cite[Theorem 2.2]{AGM}, there exist $A \in \mathcal{L} (^2 X \times Y)$, $(x_n) \subset S_X$ and $(y_n) \subset S_Y$ such that $\| A - \overline{T}\| < \eps$ and
\begin{equation}\label{eq:B_xy}
|A(x_k, y_j)| \geq \|A \| - \frac{1}{j}\ \ \text{ and }\ \ |A(x_j, y_k)|  \geq  \|A\| -\frac{1}{j}  \text{ whenever } k > j.
\end{equation}
Let $S \in \Lin (X, Y^*)$ be defined as $[Sx] (y) = A(x, y)$ for every $x \in X$ and $y \in Y$. It is clear that $\|S \| = \|A\|$ and $\| S - T \| < \eps$. To see that $S \in \wQNA (X, Y^*)$, let $(x_\alpha)$ be a subnet of $(x_n)$ such that $Sx_\alpha$ converges weak-star to some $y_0^* \in \| S\| B_{Y^*}$. By \eqref{eq:B_xy}, we can observe that
\[
|y_0^* (y_j) | =\lim_\alpha |A(x_\alpha, y_j)| \geq \|A\| - \frac{1}{j}
\]
for every $j \in \N$; hence $\|y_0^* \| = \|A\|$ which implies that $S \in \wQNA(X, Y^*)$. Similarly, let $(y_\beta)$ be a subnet of $(y_n)$ such that $S^\circ (y_\beta)$ converges weak-star to $x_0^*$ for some $x_0^* \in \|S^\circ\| B_{X^*}$. Using \eqref{eq:B_xy} again, one can similarly check that $\| x_0^*\| = \|S^\circ\|$ and so $S^\circ \in \wQNA (Y, X^*)$.
\end{proof}

We note that Theorem \ref{theorem:wQNA-dense-2} cannot be obtained directly from Theorem \ref{wQNA_dense} since there exist operators $T \in \wQNA (X,Y^*)$ for which $T^\circ \not\in \wQNA (Y, X^*)$ (see Remark~\ref{rem:circle_example}). On the other hand, unlike the case of weak-star quasi norm attaining operators, there exist Banach spaces $X$ and $Y$ such that $\QNA (X , Y^*)$ is not dense in $\Lin (X, Y^*)$.

\begin{example}\label{prop_c_0_Read}
Let $Y$ be an equivalent renorming of $c_0$ such that $Y^{**}$ is strictly convex. Then, $\overline{\QNA (c_0 , Y^{**})} \neq \Lin (c_0 , Y^{**})$.
\end{example}

\begin{proof}
We have that $\NA (c_0, Y^{**}) \subset \mathcal{F}(c_0, Y^{**})$ (see \cite[Lemma 2]{martinjfa}). Thus, the formal identity map $T$ from $c_0$ to $Y^{**}$ does not belong to $\overline{\NA (c_0, Y^{**})}$. Being a monomorphism, $T$ cannot belong to $\overline{\QNA (c_0, Y^{**})}$ by \cite[Lemma~2.2]{CCJM}.
\end{proof}

In particular, we have $\QNA ( c_0, \mathcal{R}^{**}) \neq \wQNA (c_0, \mathcal{R}^{**})$ when $\mathcal{R}$ is the Read's space (see \cite[Theorem 4]{KGM}) by using Theorem~\ref{wQNA_dense} and Example~\ref{prop_c_0_Read}. We will provide more examples of this kind in Section~\ref{section:relations}.

\section{Some remarks on weak-star quasi norm attaining operators}\label{section:some-remarks}

The main goal in this section is to provide some results on weak-star quasi norm attaining operators which can be used to better understand the concept. Some of them are analogous to the quasi norm attaining case, but there are other ones which are not. Our first result characterizes the elements of $\wQNA(X,Y^*)$ in terms of the biadjoint operators.

\begin{prop}\label{prop:basics}
Let $X$ and $Y$ be Banach spaces, and $T \in \Lin (X,Y^*)$. Then, the following are equivalent.
\begin{enumerate}
\setlength\itemsep{0.3em}
\item[\textup{(a)}] $T \in \wQNA(X,Y^*)$.
\item[\textup{(b)}] $T^{**}$ attains its norm at some $z^{**} \in S_{X^{**}}$ for which $\| T^{**} z^{**} \vert_Y  \| = \|T\|$.
\end{enumerate}
\end{prop}

\begin{proof}
$(a) \Rightarrow (b)$: Let $(x_\alpha)$ be a net in $S_X$ so that $Tx_\alpha \xrightarrow{w^*} y^* \in \|T\|S_{Y^*}$ in $Y^*$. Note that $x_\alpha \xrightarrow{w^*} z^{**}$ in $B_{X^{**}}$ for some $z^{**} \in B_{X^{**}}$, thus $Tx_\alpha \xrightarrow{w^*} T^{**} z^{**}$ in $Y^{***}$. It follows that $y^* (y) = T^{**} z^{**}  \vert_{Y} (y)$ for every $y \in Y$. Note that
$$
\|T^{**} z^{**}\|\geq \|T^{**} z^{**}\vert_Y \| = \| y^* \| = \|T\|
$$
and so $T^{**}$ attains its norm at $z^{**}$.

$(b) \Rightarrow (a)$: Let $(x_\alpha)$ be a net in $S_X$ so that $x_\alpha \xrightarrow{w^*} z^{**}$ in $X^{**}$. Passing to a subnet, we may assume that $(T x_\alpha)$ converges weak-star to some $y^* \in \|T\| B_{Y^*}$ in $Y^*$. Note also that $Tx_\alpha \xrightarrow{w^*} T^{**} z^{**}$ in $Y^{***}$. Therefore, $T^{**} z^{**} (y) = y^* (y)$ for every $y \in Y$; hence $\| T \| = \| T^{**} z^{**} \vert_Y  \| = \| y^*\|$. It follows that $T \in \wQNA(X,Y^*)$.
\end{proof}

A first consequence of the above result is that weak-star quasi norm attaining weakly compact operators are actually quasi norm attaining.

\begin{corollary}\label{weakly_compact_QNA_wQNA}
Let $X$ and $Y$ be Banach spaces, and $T \in \W (X, Y^*)$. Then the following are equivalent.
\begin{enumerate}
\setlength\itemsep{0.3em}
\item[\textup{(a)}] $T \in \wQNA(X,Y^*)$.
\item[\textup{(b)}] $T \in \QNA(X,Y^*)$.
\end{enumerate}
\end{corollary}

\begin{proof}
We only need to show $(a) \Rightarrow (b)$. Fix $T\in \wQNA(X,Y^*)$. By Proposition~\ref{prop:basics}, $T^{**}\in \NA(X^{**},Y^{***})$ and thus, being $T$ weakly compact, \cite[Proposition 5.1]{CCJM} gives that $T\in \QNA(X,Y^*)$.

An alternative proof without using neither Proposition~\ref{prop:basics} nor \cite[Proposition 5.1]{CCJM} can be given as follows. Choose a net $(x_\alpha) \subset S_X$ and $u^* \in \|T \| S_{Y^*}$ so that $Tx_\alpha \xrightarrow{w^*} u^*$. As $T \in \W (X, Y^*)$, we may choose a subnet $(x_\beta)$ such that $(Tx_\beta)$ is weakly convergent. Note that $Tx_\beta \xrightarrow{w} u^*$ since $Tx_\beta \xrightarrow{w^*} u^*$. It follows that $u^* \in \overline{T(B_X) }^{w} \cap \|T\| S_{Y^*}$; hence $T \in \QNA (X , Y^*)$.
\end{proof}

For operators whose range is a bidual space, we have the following result.

\begin{prop}\label{prop:triple}
Let $X$ and $Y$ be Banach spaces. If $T \in \mathcal{L} (X, Y^{**})$ satisfies that $T(X) \subseteq Y$ and $T^* \in \NA (Y^{***}, X^*)$, then $T \in \wQNA (X,Y^{**})$.
\end{prop}

\begin{proof}
Suppose that $\|T^* \| = \|T^* y_0^{***}\|$ for some $y_0^{***} \in S_{Y^{***}}$. Let $(x_n)$ be a sequence in $S_X$ so that $|[T^* y_0^{***}](x_n)| \longrightarrow \|T^* y_0^{***} \|$. Then
\begin{equation}\label{eq:triple}
|y_0^* (T x_n)| =|[T^* y_0^{***} ](x_n)| \longrightarrow  \|T^* y_0^{***} \|,
\end{equation}
where $y_0^* = Py_0^{***}$ and $P\colon Y^{***} \longrightarrow Y^*$ is the canonical norm one projection.
Passing to a subnet, we may assume that $(Tx_\alpha)$ converges weak-star to some $u^{**} \in \|T \| B_{Y^{**}}$. Note from \eqref{eq:triple} that $|y_0^* (u^{**})| = \|T^* \|$ which shows $\| u^{**}\| = \|T\|$. This proves that $T \in \wQNA (X, Y^{**})$.
\end{proof}

The following example shows that the converse of Proposition \ref{prop:triple} does not hold in general. That is, for a given $T \in \Lin (X, Y^{**})$, we cannot conclude that $T^* \in \NA (Y^{***}, X^*)$ under the assumptions that $T(X) \subseteq Y$ and $T \in \wQNA (X, Y^{**})$.

\begin{example}\label{prop:wQNA(c_0,ell_infty)}
There exists $T \in \wQNA ( c_0, \ell_1^* )$ such that $T(c_0) \subset c_0$ while $T^* \notin \NA (\ell_\infty^*, \ell_1)$. In particular, $\QNA (c_0, \ell_1^*) \neq \wQNA (c_0, \ell_1^*)$.
\end{example}

\begin{proof}
Consider $T \in \Lin (c_0, \ell_1^*)$ given by $Tx = (\frac{n}{n+1} x_n )_{n \in \N}$ for every $x = (x_n)_{n\in \N} \in c_0$. It is easy to check that $\|T \| = 1$. Note that
\[
T(\underbrace{1, \ldots, 1}_{n\text{-many terms}}, 0, 0, \ldots) = \left( \frac{1}{2}, \ldots, \frac{n}{n+1}, 0, 0, \ldots \right) \xrightarrow{w^*} y_0 := \left( \frac{1}{2}, \ldots, \frac{n}{n+1}, \frac{n+1}{n+2}, \ldots \right).
\]
As $\| y_0 \| = 1$, we get that $T \in \wQNA (c_0, \ell_1^*)$. However, we claim that $T^*$ does not attain its norm. Assume that $\| T^* \varphi_0\| = \|T^* \| = 1$ for some $\varphi_0 \in S_{\ell_{\infty}^*}$. If we let $u = (u_n)_{n\in \N} := \varphi \vert_{c_0} \in B_{\ell_1}$, then we have
\begin{align*}
1= \|T^* \varphi_0 \| &=  \sum_{n=1}^{\infty}\Bigl| [T^* \varphi_0] (0, \ldots, 0, \underbrace{1}_{n\text{-th term}}, 0, \ldots) \Bigr| \\
&= \sum_{n=1}^{\infty}  \Bigl| \varphi_0 (T (0, \ldots, 0, \underbrace{1}_{n\text{-th term}}, 0, \ldots) )\Bigr|  \\
&= \sum_{n=1}^{\infty} \Bigl| u  \Bigl(0, \ldots, 0, \underbrace{\frac{n}{n+1}}_{n\text{-th term}}, 0, \ldots\Bigr) \Bigr| \leq \sum_{n=1}^{\infty}  \frac{n}{n+1}| u_n | < \sum_{n=1}^\infty |u_n| \leq 1,
\end{align*}
which is a contradiction. It follows that $T^*$ does not belong to $\NA (\ell_\infty^*, \ell_1)$. The second statement follows from the fact that $T \in \QNA (X, Y)$, then $T^* \in \NA (Y^*, X^*)$ (\cite[Proposition 3.3]{CCJM}).
\end{proof}

We can extract more information on the example above. Recall that monomorphisms which quasi attain their norms are actually norm attaining operators \cite[Lemma~2.1]{CCJM}. From Example \ref{prop_c_0_Read}, we can deduce that the same result is not true for weak-star quasi norm attaining operators. Related to this observation, we have the following remark.

\begin{remark}\label{remark-monomorphism}
Notice that the above example also shows that a weak-star quasi norm attaining monomorphism is not necessarily norm attaining. Indeed, the operator $T$ considered in the proof of Example \ref{prop:wQNA(c_0,ell_infty)} satisfies that $\|T x \| \geq \|x \|/2$ for every $x \in {c_0}$. Thus, $T$ is a monomorphism and belongs to $\wQNA (c_0 , \ell_1^*)$, but $T$ does not attain its norm.
\end{remark}

For adjoint operators, the weak-star quasi norm attainment is equivalent to the quasi norm attainment and to the usual norm attainment.

\begin{remark}
Suppose that $T \in \Lin (X^*, Y^*) $ satisfies $T = S^*$ for some $S \in \Lin (Y, X)$. From the weak-star compactness of $B_{X^*}$ and the weak-star to weak-star continuity of $T$, we get
\[
T \in \NA (X^*, Y^*) \Longleftrightarrow T \in \QNA (X^*, Y^*) \Longleftrightarrow T \in \wQNA(X^*, Y^*).
\]
\end{remark}

The relation with respect to the transposition of operators is given in the following result.
Let us say that an operator $T \in \Lin (X,Y^*)$ weak-star quasi attains its norm \emph{towards} $y^*$ in the case when $y^* \in \overline{T (B_X)}^{w^*} \cap \|T \| S_{Y^*}$.

\begin{prop}\label{prop:NA_or_QNA_to_wQNA}
Let $X$ and $Y$ be Banach spaces.
\begin{enumerate}
\setlength\itemsep{0.3em}
\item[\textup{(a)}] If $T \in \QNA(X,Y^*)$, then $T^\circ \in \wQNA(Y,X^*)$.
\item[\textup{(b)}] $T \in \Lin (X, Y^*)$ satisfies $T^\circ \in \NA (Y, X^*)$ if and only if $T \in \wQNA (X, Y^*)$ towards $y^* \in \NA(Y, \mathbb{K})$.
\item[\textup{(c)}] If $Y$ is separable, then $T^\circ \in \QNA (Y, X^*)$ implies that $T \in \wQNA (X, Y^*)$.
\end{enumerate}
\end{prop}

\begin{proof}
(a): Let $(x_n)$ be a sequence in $S_X$ so that $(Tx_n)$ converges in norm to some $y_0^* \in \|T\| S_{Y^*}$. Note that $T^*$ attains its norm at $y_0^{**} \in S_{Y^{**}}$ such that $y_0^{**} (y_0^*) = \|T\|$ \cite[Proposition 3.3]{CCJM}. Let $(y_\alpha)$ be a net in $B_{Y^{**}}$ such that $y_\alpha \xrightarrow{w^*} y_0^{**}$. Passing to a subnet if necessary, we may assume that $T^\circ y_\alpha$ converges weak-star to some $x_0^* \in \|T^\circ \| B_{X^*}$. Note that
$$
[T^\circ y_\alpha](x_n) = [T x_n]( y_\alpha) \longrightarrow y_0^{**} (Tx_n).
$$
On the other hand, $[T^\circ y_\alpha](x_n)$ converges to $x_0^* (x_n)$, which implies that $y_0^{**} (Tx_n) = x_0^* (x_n)$ for every $n \in \N$. Since $(Tx_n)$ converges in norm to $y_0^*$, we can conclude that $\| x_0^* \| = \|T^\circ \|$. Thus, $T^\circ \in \wQNA (Y, X^*)$.

(b): Suppose that $\| T^\circ y_0 \| = \|T^\circ \|$ for some $y_0 \in S_Y$. Choose $(x_n) \subset S_X$ so that $|[T^\circ y_0] (x_n) | \longrightarrow \|T^\circ y_0\| = \|T^\circ \|$. As $(T x_n)\subset \| T \| B_{Y^*}$, there exists a subnet $(Tx_{\alpha} )$ that converges weak-star to some $u^* \in \|T \| B_{Y^*}$. Note that
$$| [T x_\alpha] (y_0) | = |[T^\circ y_0] (x_\alpha) | \longrightarrow \|T^\circ\|.
$$
On the other hand,
$\| u^* \| \geq |u^* (y_0)| = \lim_{\alpha} | [Tx_{\alpha}](y_0) | = \|T^\circ \|$. It follows that $T \in \wQNA(X, Y^*)$ and that $u^*$ attains its norm at $y_0$. Conversely, suppose that $T x_\alpha \xrightarrow{w^*} y^* \in \|T\| S_{Y^*} \cap \NA(Y, \mathbb{K})$ and $y^*$ attains its norm at $y_0\in S_Y$. We have $|[T^\circ y_0](x_\alpha)| = |[Tx_\alpha] (y_0)| \longrightarrow |y^* (y_0)| = \|T\|$ which implies that $T^\circ$ attains its norm at $y_0$.

(c): Assume that $Y$ is a separable Banach space and $T^\circ \in \QNA(Y, X^*)$. Choose a sequence $(y_n) \subset S_Y$ and $x_0^* \in \|T^\circ\| S_{X^*}$ such that $\| T^\circ y_n - x_0^* \| < 1/n$ for each $n \in \N$. For each $n \in \N$, choose $x_n \in S_X$ so that $|x_0^* (x_n)| > \| T \| - 1/n$. From the separability of $Y$, we get a subsequence $(x_{m_n})$ such that $T x_{m_n} \xrightarrow{w^*} u^*$ for some $u^* \in \|T\|B_{Y^*}$. For each $n \in \N$, pick $k_n \in \N$ such that
\[
| [ Tx_{m_{k_n}} - u^*  ] (y_n) | < \frac{1}{n}.
\]
Then
\begin{align*}
\|u^* \| \geq |u^* (y_n)| > | [T x_{m_{k_n}}]( y_n) | - \frac{1}{n} &= |[T^\circ y_n] ( x_{m_{k_n}}) | - \frac{1}{n} \\
&\geq |x_0^*( x_{m_{k_n}}) | - \frac{2}{n} \\
&> \|T\| - \frac{1}{m_{k_n}} - \frac{2}{n}.
\end{align*}
By letting $n \rightarrow \infty$, we obtain that $\|u^* \| = \|T\|$, as desired.
\end{proof}

\begin{remark}
There exist Banach spaces $X$ and $Y$, and $T \in \wQNA (X, Y^*)$ such that $T^\circ \notin \NA (Y, X^*)$. Indeed, let $X$ be an arbitrary Banach space and take $Y$ to be a non-reflexive Banach space. Pick $y_0^* \in S_{Y^*} \setminus \NA (Y, \mathbb{K})$ and $x_0^* \in S_{X^*} \cap \NA (X, \mathbb{K})$. Define $T \in \Lin (X, Y^*)$ by $Tx = x_0^* (x) y_0^*$ for every $x \in X$. It is clear that $T \in \NA (X, Y^*)$ which implies $T \in \wQNA (X, Y^*)$. Note that $T^\circ$ is given by $T^\circ (y) = y_0^* (y) x_0^*$ for every $y \in Y$, and so $T^\circ \notin \NA(Y,X^*)$.
\end{remark}

Furthermore, as mentioned just after Theorem \ref{theorem:wQNA-dense-2}, we will see in Remark \ref{rem:circle_example} that there exists an operator $T \in \wQNA( X,Y^*)$ such that $T^\circ \notin \wQNA( Y, X^*)$.

\section{Relations between the sets defined by different notions of the (quasi) norm attainment}
\label{section:relations}

The main aim in this section is to study when each of the containment relations in \eqref{eq:allcontents} in the introduction turns out to be an equality, and to obtain some consequences of our study. Our first result shows the relation with the reflexivity of domain or range spaces.

\begin{remark}\label{Remark_reflexivity}
Let $X$ and $Y$ be Banach spaces.
\begin{enumerate}
\setlength\itemsep{0.3em}
\item[\textup{(a)}] If $X$ is reflexive, then
$
\NA (X, Y^*) = \QNA (X, Y^*) = \wQNA (X, Y^*).
$
\item[\textup{(b)}] If $Y$ is reflexive, then
$
\QNA (X, Y^*) = \wQNA(X, Y^*).
$
\end{enumerate}
\end{remark}

\begin{proof}
(a): Assume that $X$ is reflexive and let $T \in \wQNA (X, Y^*)$ be given. Choose $(x_\alpha) \subset S_X$ and $u^* \in \|T \| S_{Y^*}$ such that $T x_\alpha \xrightarrow{w^*} u^*$. We may assume that $(x_\alpha)$ converges weakly to some $x_0 \in B_X$. As $T$ is weak to weak continuous, $T x_\alpha \xrightarrow{w} T x_0$. This shows that $Tx_0 = u^*$ and $\|Tx_0 \| = \|T \|$.

(b): If $Y$ is reflexive, then for $T \in \Lin (X, Y^*)$,
$
\overline{T(B_X)}^{w^*} = \overline{T(B_X)}^{w} = \overline{T (B_X)}^{\| \cdot \|}.
$
This proves the result.
\end{proof}

Let us comment that the equality in (a) of Remark \ref{Remark_reflexivity} is actually a characterization of the reflexivity of $X$ (see \cite[Proposition~4.1]{CCJM}). We also note that Remark \ref{Remark_reflexivity} shows the existence of a pair $(X, Y)$ of Banach spaces such that $\wQNA (X, Y^*)$ is not the same as the whole space $\Lin (X, Y^*)$. Indeed, we have $$\wQNA ( \ell_1, Y^*) = \QNA (\ell_1, Y^*) \neq \Lin (\ell_1 , Y^*)$$ for a reflexive infinite-dimensional space $Y$ (see \cite[Proposition 4.4]{CCJM}). For the next few results, we focus on the operators defined on $\ell_1$.

\begin{theorem}\label{thm:wQNA:range:ell_1}
If $Y$ is a separable non-reflexive Banach space which does not contain an isomorphic copy of $\ell_1$ \textup{(}in particular, if $Y^*$ is separable\textup{)}, then $\QNA (\ell_1, Y^{**}) \neq \wQNA ( \ell_1, Y^{**} )$.
\end{theorem}

\begin{proof}
Pick any norm one vector $y_0^{**} \in Y^{**} \setminus Y$. Since $Y$ is separable and contains no copy of $\ell_1$, each element of $B_{Y^{**}}$ is the weak-star limit of a sequence from $B_Y$ (see \cite[p.~215]{Diestel_seq_series}). Choose a sequence $(y_n)_{n\in\N} \subset B_{Y}$ such that $y_n \xrightarrow{w^*} y_0^{**}$. Let us define $T \in \Lin (\ell_1, Y^{**})$ by $Te_n = \frac{n}{n+1} y_n$ for each $n \in \N$, where $(e_n)_{n \in \N}$ denotes the canonical basis of $\ell_1$. Then, we have $\|T \| = 1$ and $Te_n \xrightarrow{w^*} y_0^{**}$ which shows $T \in \wQNA (\ell_1 , Y^{**})$.

Assume now that $T \in \QNA (\ell_1, Y^{**})$. As $T(\ell_1) \subset Y$, there exist a sequence $(x^{(n)}) \subset S_{\ell_1}$ and a vector $u \in S_Y$ such that $Tx^{(n)}\xrightarrow{\|\cdot\|} u$. Take $(z^{(n)}) \subset S_{\ell_1}$ so that $z^{(n)} = (z_j^{(n)})_{j \in \N}$ has a finite support for each $n \in \N$ and $\left\| z^{(n)} - x^{(n)} \right\| \xrightarrow[n \rightarrow \infty]{} 0$. We may assume that there exists an increasing  sequence of natural numbers $(m_n)$ such that $z_i^{(n)} = 0$ for every $i > m_n$.
Since $Tz^{(n)} \xrightarrow{\|\cdot\|} u$, we have
\begin{align*}
1 = \| u \| = \lim_{n \rightarrow \infty} \left\| T z^{(n)} \right\| &= \lim_{n \rightarrow \infty} \left\| \sum_{j=1}^{\infty} \frac{j}{j+1} z_j^{(n)} y_j \right\| \\
&\leq \lim_{n\rightarrow \infty} \sum_{j=1}^\infty \frac{j}{j+1} \left|z_j^{(n)}\right| \leq  \lim_{n\rightarrow \infty} \sum_{j=1}^\infty  \left|z_j^{(n)}\right| = 1.
\end{align*}
This implies that
\[
\lim_{n\rightarrow\infty} \sum_{j=1}^{\infty} \frac{1}{j+1} \left|z_j^{(n)}\right| = 0;
\]
hence $z_j^{(n)} \xrightarrow[n \rightarrow \infty]{} 0$ for each $j \in \N$.
Let $w^{(1)} := z^{(1)}$ and $n_1 :=1$. Choose $n_2 \in \N$ so that $\left|z_j^{(n_2)}\right| < \frac{1}{(m_{n_1})^2}$ for all $1 \leq j \leq m_{n_1}$, and define
\[
w^{(2)} := (\underbrace{0 , \ldots \ldots, 0}_{m_{n_1}\text{ many terms}}, z_{m_1+1}^{(n_2)}, \ldots, z_{m_{n_2}}^{(n_2)}, 0, 0, \ldots).
\]
Next, choose $n_3 \in \N$ so that $n_3 > n_2$ and $|z_j^{(n_3)}| < \frac{1}{(m_{n_2})^2}$ for all $1 \leq j \leq m_{n_2}$ and define
\[
w^{(3)} := (\underbrace{0 , \ldots \ldots, 0}_{m_{n_2}\text{ many terms}}, z_{m_{n_2}+1}^{(n_3)}, \ldots, z_{m_{n_3}}^{(n_3)}, 0, 0, \ldots).
\]
In this fashion, we can define vectors $w^{(k)}$ for each $k \geq 2$ and an increasing sequence of natural numbers $(n_k)_{k \geq 2}$ as
\[
w^{(k)} := (\underbrace{0 , \ldots \ldots, 0}_{m_{n_{k-1}}\text{ many terms}}, z_{m_{n_{k-1}}+1}^{(n_{k})}, \ldots, z_{m_{n_{k}}}^{(n_{k})}, 0, 0, \ldots)
\]
satisfying $|z_j^{(n_{k})}| < \frac{1}{(m_{n_{k-1}})^2}$ for $1 \leq j \leq m_{n_{k-1}}$. Note that
\[
\| w^{(k)} - z^{(n_k)} \| = |z_1^{(n_k)}| + \cdots + |z_{m_{n_{k-1}}}^{(n_k)}| <
\underbrace{\frac{1}{(m_{n_{k-1}})^2} + \cdots + \frac{1}{(m_{n_{k-1}})^2}}_{m_{n_{k-1}}\text{-many terms}} = \frac{1}{m_{n_{k-1}}} \xrightarrow[k \rightarrow \infty]{} 0;
\]
hence $Tw^{(k)} \xrightarrow{\| \cdot \|} u$ as $k \rightarrow \infty$.

On the other hand, take $u^* \in S_{Y^*}$ so that $u^* (u) = 1$. Then
\begin{align*}
u^* ( T w^{(k)} ) = u^* \left( \sum_{j=m_{n_{k-1}} + 1}^{m_{n_k}} \frac{j}{j+1}  z_j^{(n_k)} y_j \right) = \sum_{j=m_{n_{k-1}} + 1}^{m_{n_k}} \frac{j}{j+1}  z_j^{(n_k)} u^*( y_j )
 \xrightarrow[k \rightarrow \infty]{} 1.
\end{align*}
As $y_j \xrightarrow{w^*} y_0^{**}$, we get that
\[
\sup_{ m_{n_{k-1}} + 1 \leq j \leq m_{n_k} } |u^* (y_j ) - y_0^{**} (u^*)| \xrightarrow[k \rightarrow \infty]{} 0,
\]
which implies that
\[
y_0^{**} (u^*) \sum_{j=m_{n_{k-1}} + 1}^{m_{n_k}} \frac{j}{j+1}  z_j^{(n_k)}
 \xrightarrow[k \rightarrow \infty]{} 1.
\]
Thus, $|y_0^{**} (u^*)| = 1$ and
\[
\left| \sum_{j=m_{n_{k-1}} + 1}^{m_{n_k}} \frac{j}{j+1}  z_j^{(n_k)} \right|
 \xrightarrow[k \rightarrow \infty]{} 1.
\]
Put $\lambda = \lim_{k \rightarrow \infty} \sum_{j=m_{n_{k-1}} + 1}^{m_{n_k}} \frac{j}{j+1}  z_j^{(n_k)} \in S_{\C}$.
For given $y^* \in Y^*$, observe that
\begin{align*}
y^* ( T w^{(k)}) = \sum_{j=m_{n_{k-1}} + 1}^{m_{n_k}} \frac{j}{j+1}  z_j^{(n_k)} y^*( y_j ) \xrightarrow[k \rightarrow \infty]{} \lambda y_0^{**} (y^*).
\end{align*}
This means that $T w^{(k)} \xrightarrow{w^*} \lambda y_0^{**}$, so $\lambda y_0^{**}$ must be equal to $u$, which is a contradiction.
\end{proof}

\begin{remark}\label{rem:adj_NA}
Define $T \in \mathcal{L} (\ell_1, c_0^{**})$ by $Te_n = \frac{n}{n+1} y_n$ where
\[
y_n = (\underbrace{1 , \ldots \ldots, 1}_{n \text{ many terms}}, 0,0,\ldots) \in B_{c_0}
\]
for each $n \in \N$. Then the argument in the proof of Theorem \ref{thm:wQNA:range:ell_1} shows that $T \in \wQNA (\ell_1, c_0^{**})$ but $T \notin \QNA (\ell_1, c_0^{**})$. Besides, we have that $T^* \in \NA(c_0^{***}, \ell_1^*)$. Indeed, let $y^* = (a_n) \in S_{\ell_1}$ with $a_n \geq 0 $ for each $n \in \N$. Then \[
\|T^* y^* \| \geq \sup_{n \in \N} \sum_{k=1}^n \frac{n}{n+1} a_k = 1 = \|T^*\|.
\]
\end{remark}

We may read Theorem~\ref{thm:wQNA:range:ell_1} alternatively as follows: if $Y$ is a separable Banach space such that $\QNA (\ell_1, Y^{**}) = \wQNA (\ell_1, Y^{**})$, then either $Y$ is reflexive or $Y$ contains a copy of $\ell_1$. Conversely, if $Y$ is reflexive, then clearly $\QNA (\ell_1, Y^{**}) = \wQNA (\ell_1, Y^{**})$ (use Corollary~\ref{weakly_compact_QNA_wQNA} or Remark~\ref{Remark_reflexivity}). However, we see that  $Y=\ell_1$ is a separable space such that $\QNA (\ell_1, Y^{**}) \neq \wQNA (\ell_1, Y^{**})$.

\begin{example}\label{counterexample_sep_Y}
There is an operator $T \in \wQNA (\ell_1, \ell_1^{**})$ such that $T \notin \QNA (\ell_1, \ell_1^{**})$, i.e., $\QNA (\ell_1, \ell_1^{**}) \neq \wQNA (\ell_1, \ell_1^{**})$.
\end{example}

\begin{proof}
Pick a sequence $(\alpha_n) \subset (0,1)$ with $\alpha_n \rightarrow 1$. Define $T \in \Lin (\ell_1, \ell_1^{**})$ as $Tx = (\alpha_n x_n)_{n\in\N}\in\ell_1\subset \ell_1^{**}$ for every $x=(x_n)_{n\in\N} \in \ell_1$. Note that $\|T\|=1$ and $T$ does not attain its norm. First, we claim that $T \notin \QNA(\ell_1, \ell_1^{**})$. If this is not the case, there exist a sequence $(x^{(n)})_{n\in\N} \subset B_{\ell_1}$ and an element $\varphi \in S_{\ell_1^{**}}$ such that
$Tx^{(n)} \xrightarrow{\|\cdot\|} \varphi$. As $T(\ell_1) \subset \ell_1$, we see that $\varphi \in \ell_1$. Note that there exists a subnet $(x^{(\beta)} )_{\beta \in \Lambda}$ of $(x^{(n)})_{n\in\N}$ such that $x^{(\beta)} \xrightarrow{w^*} x^{(\infty)}$ for some $x^{(\infty)}\in B_{\ell_1}$ and $T$ is $w^*$-$w^*$-continuous (when $T$ is considered as an element of $\Lin (\ell_1, \ell_1)$). It follows that $T x^{(\infty)} = \varphi$, so $T$ attains its norm at $x^{(\infty)}$. This is a contradiction.

Next, we show that $T \in \wQNA (\ell_1, \ell_1^{**})$. To see this, it is enough to prove that $T^\circ \in \NA (\ell_\infty, \ell_\infty)$ (see Proposition \ref{prop:NA_or_QNA_to_wQNA}). Note that $\| T^\circ (1, 1, \ldots) \| = \| ( \alpha_1, \alpha_2, \ldots ) \| = 1$; hence $T^\circ$ attains its norm at $(1, 1, \ldots)$.
\end{proof}

\begin{remark}
If we view the operator $T$ in the proof of Example \ref{counterexample_sep_Y} as an element in $\Lin (\ell_1, c_0^*)$, then we obtain $T \in \Lin (\ell_1, c_0^*)$ such that $T^* \in \NA ( c_0^{**}, \ell_1^*)$ while $T \notin \wQNA (\ell_1, c_0^*)$.
\end{remark}

We next present some more results dealing with the equality $\QNA=\wQNA$. Recall that a Banach space $X$ is said to have the \emph{Grothendieck property} if for every sequence $(x_n^*) \subset X^*$, $x_n^* \xrightarrow{w^*} 0$ implies $x_n^* \xrightarrow{w} 0$. It is clear that if a separable Banach space $Y$ has the Grothendieck property, then $\QNA (X, Y^*) = \wQNA (X, Y^*)$. However, this is already covered by Remark \ref{Remark_reflexivity} as every separable Banach space with the Grothendieck property is reflexive. Recall also that a Banach space $X$ is called \emph{weakly compactly generated \textup{(}WCG \textup{in short)}} if there is a weakly compact set $K$ in $X$ such that $X = \overline{\text{span}} (K)$.

\begin{prop}\label{Grothendieck_WCG}
Let $X$ and $Y$ be Banach spaces. If $X$ has the Grothendieck property and $Y^*$ is a WCG space, then $\QNA (X, Y^*) = \wQNA (X, Y^*)$.
\end{prop}

\begin{proof}
It is known that if $X$ has the Grothendieck property and $Y^*$ is a WCG space, then $\Lin (X, Y^*) = \W (X, Y^*)$ (see, for instance, \cite[Exercise 13.33]{FHHMZ}). The result then follows from Corollary \ref{weakly_compact_QNA_wQNA}.
\end{proof}

It is well known that every separable Banach space is WCG. Thus, Proposition \ref{Grothendieck_WCG} in particular says that if $Y$ is a Banach space with $Y^*$ separable, then $\QNA (\ell_\infty, Y^*) = \wQNA (\ell_\infty, Y^*)$. Given a Banach space $X$, recall that its dual $X^*$ is said to have the {\it$w^*$-Kadec-Klee property} if the weak-star and norm topologies coincide on $S_{X^*}$.

\begin{prop}\label{wKK_QNA_wQNA}
Let $X$ be a Banach space. If $Y$ is a Banach space with $Y^*$ having the $w^*$-Kadec-Klee property, then $\QNA (X, Y^*) = \wQNA (X, Y^*)$.
\end{prop}

In order to prove Proposition \ref{wKK_QNA_wQNA}, we will need the following result on the $w^*$-Kadec-Klee property which is surely known, but for which we have not found a reference.

\begin{lemma}\label{wKK_subnet}
Let $Y$ be a Banach space with $Y^*$ having the $w^*$-Kadec-Klee property.
If a net $(y_\alpha^*) \subset B_{Y^*}$ converges weak-star to $u^* \in S_{Y^*}$, then there exists a subnet $(y_\beta^*)$ which converges to $u^*$ in norm.
\end{lemma}

\begin{proof}
Observe that if $\|y_\alpha^*\|$ is convergent, then $\| y_\alpha^* \|$ must converge to $1$. Indeed, if $\| y_\alpha^* \| \rightarrow \lambda < 1$, then
\[
|u^* (y) | = \lim_\alpha |y_\alpha^* (y)| \leq \lim_\alpha \| y_\alpha^* \| = \lambda
\]
for each $y \in B_Y$. Hence we have the contradiction that $\|u^* \| \leq \lambda < 1$. By compactness, we may choose a subnet $(y_\beta^*)$ such that $\| y_\beta^*\|$ is convergent. From the previous observation, $\|y_\beta^*\|$ converges to $1$.
Thus, $(\|y_\beta^*\|^{-1} y_\beta^*)$ converges weak-star to $u^*$ in $S_{Y^*}$. By the $w^*$-Kadec-Klee property of $Y^*$, $\|y_\beta^*\|^{-1} y_\beta^*$ tends to $u^*$ in norm. So, we conclude that $y_\beta^*$ converges to $u^*$ in norm.
\end{proof}

\begin{proof}[Proof of Proposition~\ref{wKK_QNA_wQNA}]
Let $T \in \wQNA (X, Y^*)$ with $\|T\| =1$ be given. Choose a net $(x_\alpha) \subset S_X$ and $y^* \in S_Y^*$ such that $T x_\alpha \xrightarrow{w^*} y^*$. By Lemma \ref{wKK_subnet}, there exists a subnet $(x_\beta) \subset S_X$ such that $T x_\beta \xrightarrow{\|\cdot\|} y^*$. It follows that $T \in \QNA (X, Y^*)$.
\end{proof}

Note that if $Y^*$ is locally uniformly rotund, then $Y^*$ has the $w^*$-Kadec-Klee property (see, for instance, \cite[Theorem 5.3.7]{M}). Moreover, it is known that $\ell_1 = c_0^*$ has the $w^*$-Kadec-Klee property \cite[Exercise 5.46]{FHHMZ}, which leads to the following consequence.

\begin{cor}\label{corollary:LURorc0}
Let $X$ be a Banach space. If $Y$ is a Banach space such that $Y^*$ is locally uniformly rotund or $Y= c_0$, then $\QNA (X , Y^*) = \wQNA (X, Y^*)$.
\end{cor}

Our last result deals with the equality between $\wQNA(X,Y^*)$ and $\Lin(X,Y^*)$.

\begin{theorem}\label{Y_sep_wQNA_ell_1_Y*}
Let $Y$ be a separable Banach space. If $\wQNA (\ell_1, Y^*) = \Lin (\ell_1, Y^*)$, then $Y$ is finite dimensional.
\end{theorem}

\begin{proof}
If $Y$ is infinite dimensional, by the Josefson-Nissenzweig theorem (see, for instance, \cite[Chapter XII]{Diestel_seq_series}), there exists a sequence $(y_n^*)_{n \in \N} \subset S_{Y^*}$ which converges weak-star to $0$. Let us define $T \in \Lin (\ell_1, Y^*)$ by $T e_n = \frac{n}{n+1} y_n^*$ for each $n \in \N$. By the assumption, we have $T \in \wQNA (\ell_1 , Y^*)$. Hence, there exists a sequence $(x^{(n)})_{n\in \N}$ such that $T x^{(n)} \xrightarrow{w^*} u^*$ for some $u^* \in S_{Y^*}$. Arguing as in the proof of Theorem \ref{thm:wQNA:range:ell_1}, we may find a sequence $(w^{(k)})_{k \in \N} \subset B_{\ell_1}$ with increasing sequences of natural numbers $(m_j)_{j\in\N}$ and $(n_j)_{j\in\N}$ satisfying
\begin{enumerate}
\item $w^{(k)} = (\underbrace{0 , \ldots \ldots, 0}_{m_{n_{k-1}}\text{ many terms}}, w_{m_{n_{k-1}}+1}^{(n_{k})}, \ldots, w_{m_{n_{k}}}^{(n_{k})}, 0, 0, \ldots)$, \label{property_(1)_of_w^k}
\item $\|w^{(k)} - x^{(n_k)} \| \xrightarrow[k \rightarrow \infty]{} 0$. \label{property_(2)_of_w^k}
\end{enumerate}
From \eqref{property_(2)_of_w^k}, we get that $T w^{(k)}  \xrightarrow{w^*} u^*$. Note from \eqref{property_(1)_of_w^k} that for $y \in Y$,
\begin{align*}
\left| [T w^{(k)}] (y) \right| \leq \sum_{j=m_{n_{k-1}} + 1}^{m_{n_k}}  \frac{j}{j+1} \left| w_j^{(n_k)} \right| \left| y_j^* (y)\right| \leq \sup_{ m_{n_{k-1}} + 1 \leq j \leq m_{n_k} } |y_j^* (y)|.
\end{align*}
As $y_j^* \xrightarrow{w^*} 0$, we can conclude that $[Tw^{(k)}](y) \xrightarrow[k \rightarrow \infty]{} 0$ for each $y \in Y$, that is, $Tw^{(k)} \xrightarrow{w^*} 0$. This implies that $u^* = 0$, a contradiction.
\end{proof}

Of course, the previous result is a characterization of finite-dimensionality (use \cite[Remark~1.4.(b)]{CCJM}, for instance).

We may also present a new proof of a result related to remotality. A bounded subset $E$ of a Banach space $X$ is said to be \emph{remotal from $x\in X$} if there is $e_x\in E$ such that $\|x-e_x\|=\sup\{\|x-e\|\colon e\in E\}$, and $E$ is said to be \emph{remotal} if it is remotal from all elements in $X$ (see \cite{Baronti-Papini} for background). Notice that in order to find a non-remotal (closed convex) subset of $X$, it suffices to find a (closed convex) subset $E$ of the open unit ball of $X$ such that $\sup\{\|e\|\colon e\in E\}=1$.

Now, observe that if $T\in \Lin(X,Y^*)\setminus \wQNA(X,Y^*)$ has norm one, then the set $K=\overline{T(B_X)}^{w^*}$ is contained in the open unit ball of $Y^*$ but $\sup_{y^* \in K} \|y^*\|=1$. Therefore, the next result on remotality for duals of separable spaces can be proved immediately using Theorem~\ref{Y_sep_wQNA_ell_1_Y*}. We note that the same statement for general real Banach spaces is proved in \cite[Corollary~2]{MartinRao}.

\begin{example}\label{example:remotal}
Let $Y$ be a (real or complex) separable infinite-dimensional Banach space. Then, there is a weak-star compact and convex subset of $Y^*$ which is not remotal.
\end{example}

Next, as a consequence of the proof of Theorem \ref{Y_sep_wQNA_ell_1_Y*}, we get that the hypothesis $y^* \in \NA(Y, \mathbb{K})$ in item (a) of Proposition~\ref{prop:NA_or_QNA_to_wQNA} cannot be removed, as the following remark shows.

\begin{remark}\label{rem:circle_example}
Let $T \in \wQNA (c_0, \ell_1^*)$ be the operator given in Example \ref{prop:wQNA(c_0,ell_infty)}, that is, $Tx = (\frac{n}{n+1} x_n )_{n \in \N}$ for every $x = (x_n)_{n\in \N} \in c_0$. Note that $T^\circ \in \Lin (\ell_1, c_0^*)$ is given by $T^\circ e_n = \frac{n}{n+1} e_n$ for each $n \in \N$. Then the proof of Theorem \ref{Y_sep_wQNA_ell_1_Y*} shows that $T^\circ \notin \wQNA (\ell_1, c_0^*)$.
\end{remark}

The same example also shows that the notion of the weak-star norm attainment depends on preduals of the range space that we are dealing with.

\begin{remark}\label{remark:depends-on-the-predual}
Consider the operator $S \in \Lin (\ell_1, \ell_1)$ given by $S e_n = \frac{n}{n+1} e_n$ for each $n \in \N$. If we consider $\ell_1=c_0^*$ in the usual way, then $(e_n)$ is weak-star convergent to $0$, hence the proof of Theorem \ref{Y_sep_wQNA_ell_1_Y*} gives that $S\notin \wQNA(\ell_1,c_0^*)$. On the other hand, if we consider $\ell_1=c^*$, then $(e_n)$ is weak-star convergent to the limit functional, which is of norm-one. Hence, $(Se_n)=\bigl(\frac{n}{n+1} e_n\bigr)$ is also weak-star convergent to the limit functional, hence $S\in \wQNA(\ell_1,c^*)$.
\end{remark}

\vspace*{0.5cm}


\noindent\textbf{Funding.}
The first author was supported by Basic Science Research Program through the National Research Foundation of Korea (NRF) funded by the Ministry of Education (2022R1A6A3A01086079). The second author was supported by NRF (NRF-2019R1A2C1003857), by POSTECH Basic Science Research Institute Grant (NRF-2021R1A6A1A10042944) and by a KIAS Individual Grant (MG086601) at Korea Institute for Advanced Study. The third author was supported by the National Research Foundation of Korea(NRF) grant funded by the Korea government(MSIT) [NRF-2020R1C1C1A01012267]. The fourth author supported by PID2021-122126NB-C31 (MCIN/AEI/FEDER, UE, AEI/10.13039/501100011033), by Junta de Andaluc\'ia I+D+i grants P20\_00255 and FQM-185, and by ``Maria de Maeztu'' Excellence Unit IMAG, reference CEX2020-001105-M funded by \phantom{---} MCIN/AEI/10.13039/501100011033.

\end{document}